\documentclass[twosided]{amsart}
\usepackage{amsmath}
\usepackage{amsfonts}
\usepackage{amssymb,enumerate}
\usepackage{amsthm}
\usepackage{hyperref}
\usepackage{centernot}
\usepackage{stmaryrd}
\theoremstyle{plain}
\newtheorem{theorem}{Theorem}[section]

\newtheorem{proposition}[theorem]{Proposition}

\begin{document}

\title{An Infinitary Generalization of the Brauer-Schur Theorem}
\author{Shahram Mohsenipour}

\address{School of Mathematics, Institute for Research in Fundamental Sciences (IPM)
        P. O. Box 19395-5746, Tehran, Iran}\email{mohseni@ipm.ir}

%\urladdr{http://www.ipm.ac.ir/IPM/people}
%\thanks{}

%\keywords{van der Waerden's theorem, infinitary version, Schur's extension}
%\date{\today}
\begin{abstract} We prove an infinitary version of the Brauer-Schur theorem.
\end{abstract}
\maketitle
\bibliographystyle{amsplain}
\section{Introduction}In \cite{thomasshahram}, the following infinitary generalization of van der Waerden's theorem is considered and proved.\\

{\em Let $2\leq k_1\leq k_2\leq k_3\leq\cdots$ be a sequence of positive integers. Then, for any finite coloring of $\mathbb{N}$, there exist positive integers $d_1,d_2,d_3,\dots$ and a color $\gamma$ such that for every $n\in\mathbb{N}$, there exists a positive integer $a_n$ such that the set
\[
\big{\{}a_n+j_1d_1+\cdots+j_nd_n|0\leq j_1\leq k_1-1,\dots,0\leq  j_n\leq k_n-1\big{\}}
\]
is monochromatic with the color $\gamma$.}\\

In \cite{infvdw}, we gave an elementary combinatorial proof. One may naturally ask whether in the above theorem, we can include the set $\{d_1,d_2,d_3,\dots\}$ in the monochromatic set. For $n=1$, this is Schur's extension of van der Waerden's theorem, which we call it here, the Brauer-Schur theorem (See Section 2.4 of Pr\"{o}mel's book \cite{promel} for some historical remarks). For finite case, namely general fixed $n$, this is actually an special case of Deuber's $(m,n,p)$ theorem. In \cite{thomasshahram}, we gave a simpler proof by direct generalization of the proof for the case $n=1$. In this paper, we show that the infinite case also holds true. More precisely, we prove\\

{\em \textbf{Main Theorem}. Let $2\leq k_1\leq k_2\leq k_3\leq\cdots$ be a sequence of positive integers. Then, for any finite coloring of $\mathbb{N}$, there exist positive integers $d_1,d_2,d_3,\dots$ and a color $\gamma$ such that for every $n\in\mathbb{N}$, there exists a positive integer $a_n$ such that the sets
\[
\big{\{}a_n+j_1d_1+\cdots+j_nd_n|0\leq j_1\leq k_1-1,\dots,0\leq  j_n\leq k_n-1\big{\}}
\]
as well as
\[
\{d_1,d_2,d_3,\dots\}
\]
are all monochromatic with the color $\gamma$.} \\

Our proof is elementary and combinatorial and consists of two steps. In Section 2, we first generalize our previous theorem from \cite{infvdw} (Theorem 3.3) and then in Section 3, we prove Theorem \ref{inductive}, from which our Main Theorem can easily be derived by decreasing induction on the number of colors. Now, we explain some notation and conventions. $\mathbb{N}$ denotes the set of positive integers. Let $c$ be a positive integer, by a $c$-coloring, we mean a coloring with $c$ many colors and furthermore, the set of the colors is $[c]=\{1,\dots,c\}$. If $b$ is a non-negative integer and $A$ is a set of integers, then $\mathbb{T}^b A$ will denote the set $\{a+b|a\in A\}$. For positive integers $k\geq 2$, and $c\geq 1$, $\textbf{W}(k,c)$ will denote the van der Waerden number which is the least integer $n$ such that for any $c$-coloring of $[n]$, there exists a monochromatic arithmetic progression of length $k$. For two finite subsets $A,B$ of $\mathbb{N}$, by $A<B$, we mean that every element of $A$ is less than every element of $B$. If $A=\{a_1,a_2,\dots,a_n\}$ with $a_1<a_2<\cdots<a_n$, then we may denote $a_t$ by $A[t]$. Also we use the abbreviation AP for the term ``arithmetic progression''.

\section{Generalized Infinitary van der Waerden's Theorem}
The following proposition is in fact a generalization of Proposition 3.1 in \cite{infvdw}.
\begin{proposition}\label{1}
Let $a\in\mathbb{N}$, $c\geq1$, $n\geq 1$, $2\leq k_1\leq k_2\leq\cdots\leq k_n$ and consider the following sequence $\langle W_i|i=0,\dots,n\rangle$ of the van der Waerden numbers with $W_0=1$ and $W_i={\bf W}(k_i,c^{W_1\cdots W_{i-1}})$ for $i=1,\dots,n$. Let $l_1,\dots,l_n$ be positive integers with $l_1>0$ and for $i=2,\dots,n$, we have $l_i>(W_1-1)l_1+\cdots+(W_{i-1}-1)l_{i-1}$. Let
\[
A_n=\big{\{}a+x_1l_1+\cdots+x_nl_n|0\leq x_1\leq W_1-1,\dots,0\leq x_n\leq W_n-1\big{\}}.
\]
Then, for any $c$-coloring of $A_n$, there exist positive integers $a_n$ and  $d_1,\dots,d_n$ with $d_i=\alpha_i l_i$ for some positive integer $\alpha_i$ with $1\leq\alpha_i\leq \lfloor\frac{W_i-1}{k_i-1}\rfloor$, for $i=1,\dots,n$, such that the following set
\[
P_n:=\big{\{}a_n+x_1d_1+\cdots+x_nd_n|0\leq x_1\leq k_1-1,\dots,0\leq x_n\leq k_n-1\big{\}}
\]
is a subset of $A_n$ and is monochromatic.
\end{proposition}
\begin{proof}
Before going into the proof, we make some simple observations. Define the finite subsets $A_0,A_1,\dots,A_{n-1}$ of $A_n$ with $A_0=\{a\}$ and for $i=1,\dots,n-1$, we have
\[
A_i=\big{\{}a+x_1l_1+\cdots+x_il_i|0\leq x_1\leq W_1-1,\dots,0\leq x_i\leq W_i-1\big{\}}.
\]
First observe that for $i=1,\dots,n$, we can represent $A_i$ as
\[
A_i=A_{i-1}\cup\mathbb{T}^{l_i}A_{i-1}\cup\mathbb{T}^{2l_i}A_{i-1}\cup\cdots\cup\mathbb{T}^{(W_i-1)l_i}A_{i-1}.
\]
Also note that the condition $l_1>0$ implies that for $A_1=\{a, a+l_1,\dots,a+(W_1-1)l_1\}$, we have $a<a+l_1<\cdots<a+(W_1-1)l_1$, so $|A_1|=W_1$. Also it is easily seen that for $i=1\dots,n$, we have $\min A_i=a$ and $\max A_i=a+(W_1-1)l_1+\cdots+(W_i-1)l_i$, so $\min\mathbb{T}^{l_{i+1}}A_i=a+l_{i+1}$ and the condition $l_{i+1}>(W_1-1)l_1+\cdots+(W_i-1)l_i$, implies $A_i<\mathbb{T}^{l_{i+1}}A_i$ which is preserved under any translation by $\mathbb{T}^{l_{i+1}}$, thus, for $i=1,\dots,n$, we get
\[
A_{i-1}<\mathbb{T}^{l_i}A_{i-1}<\mathbb{T}^{2l_i}A_{i-1}<\cdots<\mathbb{T}^{(W_i-1)l_i}A_{i-1},
\]
from which it is easily deduced that $|A_i|=W_1\cdots W_i$. Now, we start the proof which is by induction. First, let $n=1$. We have $A_1=\{a, a+l_1,\dots,a+(W_1-1)l_1\}$ with $l_1>0$. Let $\chi\colon A_1\rightarrow[c]$ be a $c$-coloring. We define a corresponding $c$-coloring $\chi^{'}\colon\{0,1,\dots,W_1-1\}\rightarrow[c]$ as follows. For $0\leq i<j\leq W_1-1$, we set $\chi^{'}(i)=\chi^{'}(j)$ iff $\chi(a+il_1)=\chi(a+jl_1)$. Recall the definition of $W_1={\bf W}(k_1,c^{W_0})={\bf W}(k_1,c)$, so by van der Waerden's theorem, there exist $\chi^{'}$-monochromatic AP of length $k_1$ in $\{0,1,\dots,W_1-1\}$, say, $b_1<\cdots<b_{k_1}$. This means that $a+b_1l_1,a+b_2l_1\dots,a+b_{k_1}l_1$ have the same $\chi$-colors. Let $d^*$ be the common difference $d^*:= b_2-b_1=\cdots=b_{k_1}-b_{k_1-1}$. Obviously $1\leq d^*\leq \lfloor\frac{W_1-1}{k_1-1}\rfloor$ and we can rewrite the above $\chi$-monochromatic sequence as $a+b_1l_1, a+b_1l_1+d^*l_1,\ldots,a+b_1l_1+(k_1-1)d^*l_1$. Now, put $a_1:=a+b_1l_1$, $d_1:=d^*l_1$ and $\alpha_1=d^*$. So, the case $n=1$ is done.

Now, suppose that $n\geq1$ and the assertion is true for $n$, we prove it for $n+1$.
%We just note that by the method of induced coloring and additive nature of the assertion, it is not only true that for any $c$-coloring of $A_n$, we can find a monochromatic subset $P_n$, but also it is true that for any $c$-coloring of any translate of $A_n$, we can find such a monochromatic subset too.
Let $\chi\colon A_{n+1}\rightarrow[c]$ be a $c$-coloring. Recall that
\[
A_{n+1}=A_{n}\cup\mathbb{T}^{l_{n+1}}A_{n}\cup\mathbb{T}^{2l_{n+1}}A_{n}\cup\cdots\cup\mathbb{T}^{(W_{n+1}-1)l_{n+1}}A_{n}
\]
as well as
\[
A_{n}<\mathbb{T}^{l_{n+1}}A_{n}<\mathbb{T}^{2l_{n+1}}A_{n}<\cdots<\mathbb{T}^{(W_{n+1}-1)l_{n+1}}A_{n},
\]
as well as
\[
|A_{n}|=|\mathbb{T}^{l_{n+1}}A_{n}|=|\mathbb{T}^{2l_{n+1}}A_{n}|=\cdots=|\mathbb{T}^{(W_{n+1}-1)l_{n+1}}A_{n}|=W_1\cdots W_n.
\]
We define the corresponding coloring $\chi^{'}\colon\{0,1,\dots,W_{n+1}-1\}\rightarrow c^{W_1\cdots W_n}$ as follows. For $0\leq i<j\leq W_{n+1}-1$, we set $\chi^{'}(i)=\chi^{'}(j)$ iff for every $1\leq t\leq W_1\cdots W_n$, we have
\[
\chi(\mathbb{T}^{il_{n+1}}A_n[t])=\chi(\mathbb{T}^{jl_{n+1}}A_n[t]).
\]
Having $W_{n+1}={\bf W}(k_{n+1},c^{W_1\cdots W_n})$, we can use van der Waerden's theorem to get a $\chi^{'}$-monochromatic AP of length $k_{n+1}$ in $\{0,1,\dots,W_{n+1}-1\}$, which we denote it by $b_1<b_2<\cdots<b_{k_{n+1}}$. Let $d^*$ be the common difference $d^*:= b_2-b_1=\cdots=b_{k_{n+1}}-b_{k_{n+1} -1}$. This means that for $1\leq i<j\leq k_{n+1}$ and $1\leq t\leq W_1\cdots W_n$, we have $\chi(\mathbb{T}^{b_il_{n+1}}A_n[t])=\chi(\mathbb{T}^{b_jl_{n+1}}A_n[t])$. By a simple computation, for $i=1,\dots,k_{n+1}$, we have
\[
\mathbb{T}^{b_il_{n+1}}A_n[t]=\mathbb{T}^{b_1l_{n+1}}A_n[t]+(b_i-b_1)l_{n+1}=\mathbb{T}^{b_1l_{n+1}}A_n[t]+(i-1)d^*l_{n+1}.
\]
Thus, we obtain that for every $1\leq t\leq W_1\cdots W_n$ and every $0\leq x_{n+1}\leq k_{n+1}-1$, we have
\[
\chi(\mathbb{T}^{b_il_{n+1}}A_n[t])=\chi(\mathbb{T}^{b_1l_{n+1}}A_n[t]+x_{n+1}d^*l_{n+1}).
\]
Now, we turn our attention to $\mathbb{T}^{b_1l_{n+1}}A_n$. In fact, $\mathbb{T}^{b_1l_{n+1}}A_n$ is represented as
\[
\big{\{}(b_1l_{n+1}+a)+x_1d_1+\cdots+x_nd_n|0\leq x_1\leq W_1-1,\dots,0\leq x_n\leq W_n-1\big{\}},
\]
therefore, we can use the induction hypothesis for it, namely, there exist $a_n,d_1,\dots,d_n$ with $d_i=\alpha_i l_i$ and $1\leq\alpha_i\leq \lfloor\frac{W_i-1}{k_i-1}\rfloor$ for $i=1,\dots,n$, such that the set
\[
\big{\{}a_n+x_1d_1+\cdots+x_nd_n|0\leq x_1\leq k_1-1,\dots,0\leq x_n\leq k_n-1\big{\}}
\]
is $\chi$-monochromatic and is included in $\mathbb{T}^{b_1l_{n+1}}A_n$. Now, we set $d_{n+1}=d^*l_{n+1}$ and define $P_{n+1}$ as
\[
\big{\{}a_n+x_1d_1+\cdots+x_{n+1}d_{n+1}|0\leq x_1\leq k_1-1,\dots,0\leq x_{n+1}\leq k_{n+1}-1\big{\}},
\]
which is $\chi$-monochromatic and is included in
\[
\mathbb{T}^{b_1l_{n+1}}A_n\cup\mathbb{T}^{b_2l_{n+1}}A_n\cup\cdots\cup\mathbb{T}^{b_{k_{n+1}}l_{n+1}}A_n
\]
which is included in $A_{n+1}$. Also it is clear that $1\leq d^*\leq \lfloor\frac{W_{n+1}-1}{k_{n+1}-1}\rfloor$, thus, we set $d^*=\alpha_{n+1}$. This completes the proof.
\end{proof}

Now, we are ready to state and prove our generalized infinitary van der Waerden's theorem.

\begin{theorem}\label{2}
Let $a\in\mathbb{N}$, $c\geq1$, and let $\langle k_i,i\in\mathbb{N}\rangle$ be a sequence of positive integers with $2\leq k_1\leq k_2\leq\cdots$ and consider the following sequence $\langle W_i|i\in\mathbb{N}\cup\{0\}\rangle$ of the van der Waerden numbers with $W_0=1$ and $W_i={\bf W}(k_i,c^{W_1\cdots W_{i-1}})$ for $i\in\mathbb{N}$. Let $l_1,l_2,\dots$ be positive integers with $l_1>0$ and for $i\geq2$, we have $l_i>(W_1-1)l_1+\cdots+(W_{i-1}-1)l_{i-1}$. Define the finite subsets $A_0,A_1,A_2\dots$, in such a way that $A_0=\{a\}$ and for $i\in\mathbb{N}$, we have
\[
A_i=\big{\{}a+x_1l_1+\cdots+x_il_i|0\leq x_1\leq W_1-1,\dots,0\leq x_i\leq W_i-1\big{\}}.
\]
Then, for any $c$-coloring of $A=\bigcup_{i\in\mathbb{N}}A_i$, there exist positive integers $a_i, d_i, i\in\mathbb{N}$ with $d_i=\alpha_i l_i$ for some positive integer $\alpha_i$ with $1\leq\alpha_i\leq \lfloor\frac{W_i-1}{k_i-1}\rfloor$, such that the following subsets
\[
\big{\{}a_i+x_1d_1+\cdots+x_id_i|0\leq x_1\leq k_1-1,\dots,0\leq x_i\leq k_i-1\big{\}}
\]
are all monochromatic with the same color.
\end{theorem}
\begin{proof}
Let $\chi\colon A\rightarrow[c]$ be a $c$-coloring. By Proposition \ref{1}, for each positive integer $m\in\mathbb{N}$, there exist $b_m$ and $d_1^m,d_2^m,\dots,d_m^{m}$ such that
\[
P_m=\big{\{}b_m+x_1d_1^m+\cdots+x_md_m^{m}|0\leq x_1\leq k_1-1,\dots,0\leq x_m\leq k_m-1\big{\}}
\]
is $\chi$-monochromatic and for $i=1,\dots,m$, we have  $d_i^m=\alpha_i^m l_m$ for some positive integer $\alpha_i^m$ with $1\leq\alpha_i^m\leq\lfloor\frac{W_m-1}{k_m-1}\rfloor$. This means that for any positive integer $i$, all $d_i^i,d_i^{i+1},d_i^{i+2},d_i^{i+3},\dots$ are bounded by $(W_i-1)l_i$. Now by a simple iterated use of pigeonhole argument, we obtain a color $\gamma\in[c]$ and a strictly increasing sequence of positive integers $n_1<n_2<n_3<n_4<\cdots$ such that $P_{n_1},P_{n_2},P_{n_3},P_{n_4},\dots$ are all monochromatic with color $\gamma$ and also we have
\begin{eqnarray*}
                      d_1^{n_1}=d_1^{n_2}=d_1^{n_3}=d_1^{n_4}=&\cdots&\\
                                d_2^{n_2}=d_2^{n_3}=d_2^{n_4}=&\cdots&\\
                                          d_3^{n_3}=d_3^{n_4}=&\cdots&\\
                                                    d_4^{n_4}=&\cdots&\\
                                                              &\cdots&
\end{eqnarray*}
Now, for $i\in\mathbb{N}$, we put $a_i:=b_{n_i}$ and $d_i:=d_i^{n_i}$, which finishes the proof.
\end{proof}

We now introduce a very useful notation. Let for $i\in\mathbb{N}$, $a_i,d_i$ be positive integers. Suppose also that for positive integers $k_i,i\in\mathbb{N}$, we have $2\leq k_1\leq k_2\leq\dots$. For $i\in\mathbb{N}$, set
\[
A_i=\big{\{}a_i+x_1 d_1+\cdots+x_i d_i|0\leq x_1\leq k_1-1,\dots,0\leq x_i\leq k_i-1\big{\}},
\]
and set $A=\bigcup_{i\in\mathbb{N}}A_i$. We denote $A$ as a triple of sequences, from which $A$ and $A_i$ can be reproduced in the obvious way:
\[\begin{cases}
A^1=\langle a_1,a_2,a_3,\dots\rangle,\\
A^2=\langle d_1,d_2,d_3,\dots\rangle,\\
A^3=\langle k_1,k_2,k_3,\dots\rangle.
\end{cases}
\]
Then, we can reformulate Theorem \ref{2} as follows.\\

\noindent\textbf{Theorem 2.2.} {\em Let $a, c, k_i, W_i, l_i, i\in\mathbb{N}$, satisfying the conditions of Theorem \ref{2}. Suppose that $\chi\colon A\rightarrow[c]$ is a $c$-coloring, where $A$ is given by
\[
\begin{cases}
A^1=\langle a,a,a,\dots\rangle,\\
A^2=\langle l_1,l_2,l_3,\dots\rangle,\\
A^3=\langle W_1,W_2,W_3,\dots\rangle.
\end{cases}
\]
Then, for $i\in\mathbb{N}$, there exist positive integers $a_i, d_i$, with $d_i=\alpha_i l_i$ for some positive integer $\alpha_i$ with $1\leq\alpha_i\leq \lfloor\frac{W_i-1}{k_i-1}\rfloor$ such that $B\subseteq A$ and $B$ is $\chi$-monochromatic, where $B$ is given by the triplet
\[
\begin{cases}
B^1=\langle a_1,a_2,a_3,\dots\rangle,\\
B^2=\langle d_1,d_2,d_3,\dots\rangle,\\
B^3=\langle k_1,k_2,k_3,\dots\rangle.
\end{cases}
\]}

\section{Infinitary Brauer-Schur Theorem}

%For the convenience of the reader we first prove it for 2 colors and then for 3 colors which hint how the general pattern work.
%\begin{proposition}
%Let $a\in\mathbb{N}\cup\{0\}$ and let $\{k_i,i\in\mathbb{N}\}$ be a sequence of positive integers with $2\leq k_1\leq k_2\leq\cdots$ and consider the following sequence $\{W_i|i\in\mathbb{N}\cup\{0\}\}$ of van der Waerden numbers according to the square sequence $k_1^2\leq k_2^2\leq\cdots$, namely, $W_0=1$ and $W_i={\bf W}(k_i^2,2^{W_1\cdots W_{i-1}})$ for $i\in\mathbb{N}$. Let $l_1,\dots,l_n,\dots$ be positive integers with $l_1>0$ and for $i\geq2$, we have $l_i>(W_1-1)l_1+\cdots+(W_{i-1}-1)l_{i-1}$. Define the finite subsets $A_0,A_1,\dots,A_n,\dots$ of non-negative integers with $A_0=\{a\}$ and for $i\in\mathbb{N}$, we have
%\[
%A_i=\big{\{}a+x_1l_1+\cdots+x_il_i|0\leq x_1\leq W_1-1,\dots,0\leq x_i\leq W_i-1\big{\}}.
%\]
%Also suppose for $i\in\mathbb{N}$
%\[
%B_i=\big{\{}x_1l_1+\cdots+x_il_i|0\leq x_1\leq k_1-1,\dots,0\leq x_i\leq k_i-1\big{\}}.
%\]
%Let $A=\bigcup_{i\in\mathbb{N}}A_i$ and $B=\bigcup_{i\in\mathbb{N}}B_i$, then for any $2$-coloring of $A\cup B\setminus\{0\}$, there exist non-negative integers $a_i, i\in\mathbb{N}$ and positive integers $d_i,$ $i\in\mathbb{N}$ such that the following subsets
%\[
%\big{\{}a_i+x_1d_1+\cdots+x_id_i|0\leq x_1\leq k_1-1,\dots,0\leq x_i\leq k_i-1\big{\}}
%\]
%as well as the set $\{d_i|i\in\mathbb{N}\}$ are all monochromatic with the same color.
%\end{proposition}

The proof of Main Theorem, as in the case of the classical Brauer-Schur theorem and also our previous finite generalization of it, proceeds by induction on the number of colors. In fact, the following theorem allows us to do a decreasing induction on the number of colors. We leave it to the reader to observe that Main Theorem is easily implied.

\begin{theorem}\label{inductive}
Let $c\geq1$, and $a\in\mathbb{N}$ and let $\langle k_i,i\in\mathbb{N}\rangle$ be an \textsf{unbounded} sequence of positive integers with $2\leq k_1\leq k_2\leq\cdots$ and consider the following sequence $\langle W_i|i\in\mathbb{N}\cup\{0\}\rangle$ of the van der Waerden numbers according to the square sequence $k_1^2\leq k_2^2\leq\cdots$, where $W_0=1$ and $W_i={\bf W}(k_i^2,c^{W_1\cdots W_{i-1}})$ for $i\in\mathbb{N}$. Suppose that $l_1,\dots,l_n,\dots$ are positive integers with $l_1>0$ and for $i\geq2$, we have $l_i>(W_1-1)l_1+\cdots+(W_{i-1}-1)l_{i-1}$. Also Suppose that $\chi\colon A\cup B\rightarrow[c]$ is a $c$-coloring, where $A,B$ are given by
\[
\begin{cases}
A^1=\langle a,a,a,\dots\rangle,\\
A^2=\langle l_1,l_2,l_3,\dots\rangle,\\
A^3=\langle W_1,W_2,W_3,\dots\rangle,
\end{cases}
\begin{cases}
B^1=\langle 0,0,0,\dots\rangle,\\
B^2=\langle l_1,l_2,l_3,\dots\rangle,\\
B^3=\langle W_1,W_2,W_3,\dots\rangle.
\end{cases}
\]
Then, $\underline{either}$ there exist positive integers $a_i,d_i,i\in\mathbb{N}$ such that $C\subseteq A$, $\{d_i|i\in\mathbb{N}\}\subseteq B$ and $C\cup\{d_i|i\in\mathbb{N}\}$ is $\chi$-monochromatic with a color, say $c$, where $C$ is given by
\[
\begin{cases}
C^1=\langle a_1,a_2,a_3,\dots\rangle,\\
C^2=\langle d_1,d_2,d_3,\dots\rangle,\\
C^3=\langle k_1,k_2,k_3,\dots\rangle,
\end{cases}
\]
$\underline{or}$ there exist positive integers $b,l_i^*,i\in\mathbb{N}$ with $l_1^*>0$ and for $i\geq2$, we have
\[
l_i^*>(W_1-1)l_1^*+\cdots+(W_{i-1}-1)l_{i-1}^*,
\]
such that for $D,E$ defined below
\[
\begin{cases}
D^1=\langle b,b,b,\dots\rangle,\\
D^2=\langle l_1^*,l_2^*,l_3^*,\dots\rangle,\\
D^3=\langle W_1,W_2,W_3,\dots\rangle,
\end{cases}
\begin{cases}
E^1=\langle 0,0,0,\dots\rangle,\\
E^2=\langle l_1^*,l_2^*,l_3^*,\dots\rangle,\\
E^3=\langle W_1,W_2,W_3,\dots\rangle,
\end{cases}
\]
we have $D\cup E\subseteq B$ and the range of the restriction of $\chi$ to $D\cup E$ is a subset of $[c-1]$.
\end{theorem}

\begin{proof}
For $c=1$, the theorem is trivial. We assume $c\geq2$. By Theorem \ref{2} (the reformulated version), there exist positive integers $b_1,b_2,b_3,\dots$ and $d_1^*,d_2^*,d_3^*,\dots$ with $C_1\subseteq A$ and $C_1$ is $\chi$-monochromatic with a color, say $c$, where $C_1$ is given by the triplet
\[
\begin{cases}
C_1^1=\langle b_1,b_2,b_3,\dots\rangle,\\
C_1^2=\langle d_1^*,d_2^*,d_3^*,\dots\rangle,\\
C_1^3=\langle k_1^2,k_2^2,k_3^2,\dots\rangle.
\end{cases}
\]
Note that for $i\in\mathbb{N}$, we have $d_i^*=\alpha_i l_i$ with $1\leq\alpha_i\leq \lfloor\frac{W_i-1}{k_i-1}\rfloor$. Hence, for $C_2$ given by
\[
\begin{cases}
C_2^1=\langle 0,0,0,\dots\rangle,\\
C_2^2=\langle d_1^*,d_2^*,d_3^*,\dots\rangle,\\
C_2^3=\langle k_1^2,k_2^2,k_3^2,\dots\rangle,
\end{cases}
\]
we also have $C_2\subseteq B$ and therefore, all of its elements receive a $\chi$-color.

There are two cases corresponding to the two cases of the theorem:

\underline{Case 1}: There exist infinitely many non-empty finite subsets $I_1<I_2<I_3<\cdots$ of $\mathbb{N}$, with the following property. For each $i\in\bigcup_{j\in\mathbb{N}} I_j$, there exists $x_i^*$ with $0< x_i^*\leq k_i-1$ such that
\[
\chi\big{(}\sum_{i\in I_1} x_i^*d_i^*\big{)}=c,\,\,\, \chi\big{(}\sum_{i\in I_2} x_i^*d_i^*\big{)}=c,\,\,\,\chi\big{(}\sum_{i\in I_3} x_i^*d_i^*\big{)}=c,\,\,\,\ldots.
\]
Notice that this is well-defined, since all $\sum_{i\in I_1} x_i^*d_i^*$, $\sum_{i\in I_2} x_i^*d_i^*$, $\sum_{i\in I_3} x_i^*d_i^*$, etc, are elements of $C_2$ and therefore receive a $\chi$-color.  Now, we define the desired $d_i,i\in\mathbb{N}$ as follows.
\[
d_1=\sum_{i\in I_1} x_i^*d_i^*,\,\,\,d_2=\sum_{i\in I_2} x_i^*d_i^*,\,\,\,d_3=\sum_{i\in I_3} x_i^*d_i^*,\,\,\,\ldots.
\]
For $i\in\mathbb{N}$, we put $\min I_i=n_i$ and $\max I_i=m_i$. Obviously for each $t\in\mathbb{N}$, we have $n_t\geq t$ and therefore, for each $i\in I_t$, we have $t\leq n_t\leq i$.
Now, we claim that for any integers $x_1,\dots,x_t$ with $0\leq x_1\leq k_1-1$,$\dots$, $0\leq x_t\leq k_t-1$, there are integers $y_i$ for $i\in I_1\cup\cdots\cup I_t$, with $0\leq y_i\leq k_i^2-1$ such that
\[
\sum_{i=1}^{t}x_id_i\,\,\,\,=\sum_{i\in I_1\cup\cdots\cup I_t}y_id_i^*.
\]
To see this, we have that
\begin{eqnarray*}
     \sum_{i=1}^{t}x_id_i&=&\sum_{i=1}^{t}x_i(\sum_{j\in I_i} x_j^*d_j^*)=\sum_{i=1}^{t}\sum_{j\in I_i}( x_ix_j^*)d_j^*\\
     &=&\sum_{j\in I_1}(x_1x_j^*)d_j^*+\sum_{j\in I_2}(x_2x_j^*)d_j^*+\cdots+\sum_{j\in I_t}(x_tx_j^*)d_j^*.
\end{eqnarray*}
Note that for $1\leq i\leq t$ and $j\in I_i$, we have $j\geq i$, thus the following inequalities hold.
\[
\forall j\in I_1: 0\leq x_1x_j^*\leq(k_1-1)(k_j-1)\leq (k_j^2-1),
\]
\[
\forall j\in I_2: 0\leq x_2x_j^*\leq(k_2-1)(k_j-1)\leq (k_j^2-1),
\]
\[
\dots\dots\dots
\]
\[
\forall j\in I_t: 0\leq x_tx_j^*\leq(k_t-1)(k_j-1)\leq (k_j^2-1).
\]
Now, for $j\in I_1$, we put $y_j=x_1x_j^*$ and for $j\in I_2$, we put $y_j=x_2x_j^*$, $\dots$, and for $j\in I_t$, we put $y_j=x_tx_j^*$. This proves the claim. Now, from $I_1<I_2<\cdots I_t$, we infer that $m_t=\max I_t$ is the maximum element of $I_1\cup\cdots\cup I_t$, so according to the claim, we have the following inclusions (recall $b_1,b_2,\dots$):
\[
\big{\{}b_{m_t}+x_1d_1+\cdots+x_td_t|0\leq x_1\leq k_1-1,\dots,0\leq x_t\leq k_t-1\big{\}}
\]
\[
\subseteq\big{\{}b_{m_t}+\sum_{i\in I_1\cup\cdots\cup I_t}y_id_i^*|0\leq y_i\leq k_i^2-1\big{\}}
\]
\[
\subseteq\big{\{}b_{m_t}+z_1d_1^*+\cdots+z_{m_t}d_{m_t}^*|0\leq z_1\leq k_1^2-1,\dots,0\leq z_{m_t}\leq k_{m_t}^2-1\big{\}}
\]
\[
\subseteq (C_1)_{m_t}.
\]
Now, put $a_1=b_{m_1}$,$a_2=b_{m_2}$,$a_3=b_{m_3}$ and therefore, for $C$ given by
\[
\begin{cases}
C^1=\langle a_1,a_2,a_3,\dots\rangle=\langle b_{m_1},b_{m_2},b_{m_3},\dots\rangle,\\
C^2=\langle d_1,d_2,d_3,\dots\rangle,\\
C^3=\langle k_1,k_2,k_3,\dots\rangle,
\end{cases}
\]
we get that $C\subseteq C_1$ as well as $C\cup\{d_1,d_2,\dots\}$ is $\chi$-monochromatic with the color $c$ and consequently, we have proved the theorem in Case 1.

\underline{Case 2}: Suppose that Case 1, doesn't occur. This means that there is $p\in\mathbb{N}$, such that all elements of the following subset $G$ of $C_2$, get their $\chi$-colors from $\{1,\dots,c-1\}$.
\[
G=\big{\{}x_pd_p^*+\cdots+x_{p+t}d_{p+t}^*|t\in\mathbb{N},0\leq x_p\leq k_p-1,\dots,0\leq x_{p+t}\leq k_{p+t}-1\big{\}}.
\]
Now, we pass to the following subsets $D_1,E_1$ of $G$ given by
\[
\begin{cases}
D_1^1=\langle d_p^*,d_p^*,d_p^*,\dots\rangle\\
D_1^2=\langle d_{p+1}^*,d_{p+2}^*,d_{p+3}^*,\dots\rangle,\\
D_1^3=\langle k_{p+1},k_{p+2},k_{p+3},\dots\rangle,
\end{cases}
\begin{cases}
E_1^1=\langle 0,0,0,\dots\rangle\\
E_1^2=\langle d_{p+1}^*,d_{p+2}^*,d_{p+3}^*,\dots\rangle,\\
E_1^3=\langle k_{p+1},k_{p+2},k_{p+3},\dots\rangle.
\end{cases}
\]
For getting $D,E,$ satisfying the conclusion of the theorem, we first define a sequence of triples of integers $\{\langle t_i,m_i,l_i^*\rangle|i\in\mathbb{N}\}$ inductively, as follows. Set $t_1=0$ and let $m_1$ be the least integer with $m_1>p$ and $k_{m_1}\geq W_1$. This is possible as the sequence $k_1,k_2,k_3,\dots$ is unbounded. Now put $l_1^*=d_{m_1}^*$ and let $m_2$ be the least integer with $m_2>m_1+t_1=m_1$, such that $k_{m_2}\geq W_2$ and let $t_2$ be the least integer with
\[
d_{m_2}^*+d_{m_2+1}^*+\cdots+d_{m_2+t_2}^*>(W_1-1)l_1^*.
\]
Then, put $l_2^*=d_{m_2}^*+d_{m_2+1}^*+\cdots+d_{m_2+t_2}^*$. To continue, let $m_3$ be the least integer with $m_3>m_2+t_2$ such that $k_{m_3}\geq W_3$ and let $t_3$ be the least integer with
\[
d_{m_3}^*+d_{m_3+1}^*+\cdots+d_{m_3+t_3}^*>(W_1-1)l_1^*+(W_2-1)l_2^*.
\]
Put $l_3^*=d_{m_3}^*+d_{m_3+1}^*+\cdots+d_{m_3+t_3}^*$ and so on. Clearly this can be continued and we obtain the sequence $\langle l_i^*|i\in\mathbb{N}\rangle$ with $l_1^*>0$ and for $i\geq2$, we have
\[
l_i^*>(W_1-1)l_1^*+\cdots+(W_{i-1}-1)l_{i-1}^*.
\]
Also for $i\in\mathbb{N}$, we have $k_{m_i}\geq W_i$. Now, put $b=d_p^*$ and observe that the construction of $\{\langle t_i,m_i,l_i^*\rangle|i\in\mathbb{N}\}$, implies (by an easy computation) that $D_2\subseteq D_1$ and $E_2\subseteq E_1$, where $D_2,E_2$ are given by
\[
\begin{cases}
D_2^1=\langle b,b,b,\dots\rangle\\
D_2^2=\langle l_1^*,l_2^*,l_3^*,\dots\rangle,\\
D_2^3=\langle k_{m_1},k_{m_2},k_{m_3},\dots\rangle,
\end{cases}
\begin{cases}
E_2^1=\langle 0,0,0,\dots\rangle\\
E_2^2=\langle l_1^*,l_2^*,l_3^*,,\dots\rangle,\\
E_2^3=\langle k_{m_1},k_{m_2},k_{m_3},\dots\rangle.
\end{cases}
\]
Now, for the final step of the proof, note that in the third sequences $D_2^3,E_2^3$, above, we can actually put smaller elements and pass to subsets. Hence, we get the desired $E,D$ as follows
\[
\begin{cases}
D^1=\langle b,b,b,\dots\rangle\\
D^2=\langle l_1^*,l_2^*,l_3^*,\dots\rangle,\\
D^3=\langle W_1,W_2,W_3,\dots\rangle,
\end{cases}
\begin{cases}
E^1=\langle 0,0,0,\dots\rangle\\
E^2=\langle l_1^*,l_2^*,l_3^*,,\dots\rangle,\\
E^3=\langle W_1,W_2,W_3,\dots\rangle,
\end{cases}
\]
which completes the proof of the theorem in Case 2.
\end{proof}

\subsection*{Acknowledgment} The research of the author was in part supported by a grant from IPM (No. 02030403).

\bibliography{reference}
\bibliographystyle{plain}
\end{document}